\documentclass[12pt]{amsart}
\usepackage{amssymb}
\newcommand{\K}{\mathbb{C}}

\newcommand{\g}{\mathfrak{g}}

\newcommand{\Z}{\mathbb{Z}}

\newcommand{\Dcal}{\mathcal{D}}
\newcommand{\Str}{\mathcal{O}}
\newcommand{\h}{\mathfrak{h}}

\newcommand\Ext{\operatorname{Ext}}

\newcommand\GL{\operatorname{GL}}

\newcommand\m{\mathfrak{m}}

\newcommand\End{\operatorname{End}}

\newcommand\Hom{\operatorname{Hom}}

\newcommand\quo{/\!/}

\newcommand\SL{\operatorname{SL}}
\newcommand\z{\mathfrak{z}}
\newcommand\W{\mathbb{A}}

\newcommand\Aa{\mathbb{A}}

\newcommand\param{\mathfrak{c}}
\newcommand\Halg{\mathbf{H}}
\newcommand\Pro{\mathcal{P}}

\newcommand\ZZ{\mathbb{Z}}

\newcommand{\HH}{\operatorname{HH}}
\newcommand{\Drm}{\operatorname{D}}
\newcommand{\Fi}{\mathbb{F}}
\newcommand{\Az}{\mathrm{O}}
\newtheorem{Thm}{Theorem}[section]
\newtheorem{Prop}[Thm]{Proposition}

\newtheorem{Lem}[Thm]{Lemma}
\theoremstyle{definition}

\newtheorem{Rem}[Thm]{Remark}

\numberwithin{equation}{section}
\author{Ivan Losev}
\title{On Procesi bundles}
\address{Department
of Mathematics, Northeastern University, Boston MA 02115 USA}
\email{i.loseu@neu.edu}
\begin{document}
\begin{abstract}
Procesi bundles are certain vector bundles  on symplectic resolutions of  quotient singularities
for wreath-products of the symmetric groups with the Kleinian groups. Roughly speaking, we can define Procesi bundles as bundles on
resolutions that provide derived McKay equivalence. In this paper we classify Procesi bundles on resolutions obtained
by Hamiltonian reduction and relate the Procesi bundles to the tautological bundles on the resolutions.
Our proofs are based on  deformation arguments and a connection of Procesi bundles with symplectic reflection
algebras.
\end{abstract}
\maketitle
\section{Introduction}

\subsection{Procesi bundles}
Let $\Gamma_1\subset \SL_2(\K)$ be a finite subgroup (a so called Kleinian group).
Such groups are in one-to-one correspondence with simply laced Dynkin diagrams.
Fix $n\geqslant 1$. Then we have the semidirect product $\Gamma_n=\mathfrak{S}_n\ltimes \Gamma_1^n$
that naturally acts on $V:=\K^{2n}$ by linear symplectomorphisms. So we can form the quotient
variety $X_0:=V/\Gamma_n$. This variety has an action of $\K^\times$ by dilations.

Now let $X$ be a $\K^\times$-equivariant resolution of singularities of $X_0$ that is symplectic
in the sense that there is a symplectic form on $X$ extending the form on the regular locus of $X_0$.
Let $\pi:X\rightarrow X_0=V/\Gamma_n$ denote the resolution morphism, and $\eta:V\rightarrow V/\Gamma_n$
be the quotient morphism. We remark that $\pi^*$ identifies $\K[X_0]=\K[V]^{\Gamma_n}$ with $\K[X]$.

Consider the algebra $\K[V]\#\Gamma_n$ that is a skew-group algebra for the action of $\Gamma_n$ on $\K[V]$.
We remark that the algebra $\K[X_0]=\K[V]^{\Gamma_n}$ is naturally included into $\K[V]\#\Gamma_n$
and coincides with the center of the latter. Also $\K[V]\#\Gamma_n$ is naturally graded with $V^*$
being of degree $1$ and $\Gamma_n$ being of degree $0$.

By a {\it Procesi bundle} on $X$ we mean a $\K^\times$-equivariant vector bundle $\Pro$ together with a $\K^\times$-equivariant identification $\End(\Pro)\xrightarrow{\sim} \K[V]\#\Gamma_n$ of $\K[X]=\K[V]^{\Gamma_n}$-algebras
such that $\Ext^i(\Pro,\Pro)=0$ for $i>0$. Two  Procesi bundles $\Pro^1,\Pro^2$ are said to be equivalent
if there is a $\K^\times$-equivariant isomorphism $\Pro^1\xrightarrow{\sim}\Pro^2$ of vector bundles on $X$ such that
the induced automorphism of $\K[V]\#\Gamma_n$ is inner (and from the degree consideration is the conjugation by
an invertible element of $\K\Gamma_n$).

We remark that an isomorphism $\End(\Pro)\xrightarrow{\sim}\K[V]\#\Gamma_n$ equips $\Pro$ with a fiberwise
action of $\Gamma_n$. It is easy to see that every fiber is a regular representation of $\Gamma_n$. So $\Pro^{\Gamma_n}$
is a line bundle. By a {\it normalized Procesi bundle} we mean a Procesi bundle $\Pro$ with $\Pro^{\Gamma_n}=\Str_X$.
If $\Pro$ is a Procesi bundle, then using the natural identification  $\End(\Pro)\cong \End(\mathcal{L}\otimes \Pro)$,
where $\mathcal{L}$ is a line bundle, we equip $\mathcal{L}\otimes \Pro$ with the structure of a Procesi bundle.
So any Procesi bundle $\Pro$ can be normalized by tensoring it with $(\Pro^{\Gamma_n})^{-1}$.

We remark that two non-equivalent Procesi bundles can be isomorphic as $\K^\times$-equivariant vector bundles.
Indeed, one can twist the fiberwise action of $\Gamma_n$ by tensoring it with a one-dimensional representation.
Note, however, that if $\Pro^1,\Pro^2$ are two normalized Procesi bundles, then any $\K^\times$-equivariant
isomorphism $\Pro^1\xrightarrow{\sim}\Pro^2$ will be an equivalence of Procesi bundles, we will check this rigorously below.


In the special case when $n=1$ the Procesi bundles are easy to produce. Namely, $X$ has to be the minimal resolution
of $X_0$. It is well-known that $X$ can be constructed as a moduli space of $\K[V]\#\Gamma_1$-modules
that are isomorphic to $\K\Gamma_1$ as $\Gamma_1$-modules and are subject to a (general enough) stability condition. Because
of this realization,  $X$ comes
equipped with  a tautological bundle of rank $|\Gamma_1|$ and this bundle is a Procesi bundle in the sense
explained above, as was checked in \cite{KV}. We remark that the number of non-equivalent stability conditions
equals $|W|$, where $W$ is the Weyl group of the Dynkin diagram corresponding to $\Gamma_1$.

Outside of the $n=1$ case, the Procesi bundles are hard to construct. The first known case was
for $\Gamma_1=\{1\}$, where $X$ is the Hilbert scheme $\operatorname{Hilb}_n$ of
$n$ points on $\K^2$. There one can also consider the tautological bundle $\mathcal{T}$ but its
rank is $n$ instead of $n!$. A Procesi bundle $\mathcal{P}$ on $X$ was constructed by Haiman in
\cite{Haiman} and was used to prove the Macdonald positivity conjecture. This bundle $\mathcal{P}$
is still related to $\mathcal{T}$: namely, $\mathcal{P}^{\mathfrak{S}_{n-1}}\cong \mathcal{T}$.
An alternative construction of $\mathcal{P}$ was later given by Ginzburg, \cite{Ginzburg_Pro}.
Let us remark that $\Pro^*$ is also a Procesi bundle (thanks to the isomorphism of
$\K[V]\#\Gamma_n$ with its opposite given by the identity map on $V$ and the inversion on $\Gamma_n$).
It is easy to see that $\Pro$ is not equivalent to $\Pro^*$ and so we, at least, have two different
Procesi bundles on $X=\operatorname{Hilb}_n$.

In \cite{BK}, Bezrukavnikov and Kaledin proved that a Procesi bundle exists on any $X$.
We will describe all Procesi bundles on  $X$ provided $X$ is obtained by Hamiltonian reduction
as explained in the next subsection. Also we confirm a part of \cite[Conjecture 7.2.13]{Haiman_CDM}
 and investigate some other properties of Procesi bundles. Yet another property
was recently established by Bezrukavnikov and Finkelberg, \cite{BF}: they checked certain
triangularity properties of Procesi bundles in the case when $\Gamma_1$ is cyclic and generalized
the Macdonald positivity to that setting.

We also would like to point out that the notion of a Procesi bundle still makes sense if we replace
$\K$ with an algebra $R$ over a suitable algebraic extension of $\Z$, compare with \cite{BK}.

\subsection{Resolutions via Hamiltonian reduction}\label{SS_Ham_res}
 One can construct a symplectic resolution $X$ of $X_0:=V/\Gamma_n$ using Hamiltonian
reduction. In this paper we are going to deal with these resolutions, conjecturally
there is nothing else.

We consider the affine Dynkin quiver $Q$ corresponding to $\Gamma_1$ (we have a single
loop if $\Gamma_1=\{1\}$). Let $0$ denote the extending vertex and $\delta$ be the indecomposable
imaginary root. Consider the dimension vector $v:=n\delta$ and the framing vector $w:=\epsilon_0$
(the coordinate vector at the extending vertex). Then we can form the representation space
\begin{align*}U_0:=\operatorname{Rep}(Q,v,w)=&\bigoplus_{a\in Q_1} \Hom_{\K}(\K^{v_{t(a)}}, \K^{v_{h(a)}})\oplus
\bigoplus_{i\in Q_0}\Hom_{\K}(\K^{w_i}, \K^{v_i})\\& =\bigoplus_{a\in Q_1} \Hom_{\K}(\K^{n\delta_{t(a)}}, \K^{n\delta_{h(a)}})\oplus \K^n
\end{align*}
acted on by $G=\GL(n\delta):=\prod_{i\in Q_0}\GL(n\delta_i)$.
Set $U:=T^* U_0$ and consider a generic character $\theta$ of $G$ (``generic'' means that the action
of $G$ on $U^{ss}$ is free). The $G$-action is symplectic and so  we have a quadratic moment map $\mu:U\rightarrow \g$.
We have a $\K^\times$-equivariant identification $X_0\cong \mu^{-1}(0)\quo G$ (the actions
come from the dilation actions on $U,V$) of Poisson algebraic varieties. Also we have a $\K^\times$-equivariant
symplectic resolution $X^\theta:=\mu^{-1}(0)^{\theta-ss}/G\twoheadrightarrow X_0$.  This resolution
comes equipped with a tautological bundle $\mathcal{T}^\theta$: it is obtained from the equivariant
bundle with fiber $\bigoplus_{i\in Q_0} \K^{n\delta_i}\otimes \K^{\delta_i*}$ (where $G$ acts on the first
factor) on $V$ by descent.

When $n=1$, we get the minimal resolution of the Kleinian singularity, while for $\Gamma_1$ we get the Hilbert
scheme.

Let us explain what values of $\theta$ are generic. We can embed the character group $\operatorname{Hom}(G,\K^\times)$
into the dual $\h^*$ of the Cartan subalgebra of the Kac-Moody algebra $\g(Q)$ associated to the quiver $Q$: we identify
the character $\det_i$ of taking $\det$ of the $i$th component with the fundamental weight $\omega_i$. Then the non-generic
values of $\theta$ are precisely those that vanish on a positive root $\alpha<n\delta$. A connected component
of the complement of these hyperplanes in $\h^*(\mathbb{R})$ will be called a chamber.

Of course, the resolutions corresponding to stability conditions in the same chamber are the same. But also different
chambers may give rise to the same resolution. Namely, let $W$ denote the Weyl group of the Dynkin diagram
of $\Gamma_1$. Then $W$ naturally acts on $\h^*$. Also the group $\ZZ/2\ZZ$ acts on $\h^*$: it fixes all finite roots
and sends $\delta$ to $-\delta$. The resolutions corresponding to $W$-conjugate stability conditions $\theta$
are the same, see \cite{Maffei}. This also should be true for the $\ZZ/2\ZZ$-action but we do not check this.
Finally, the resolutions corresponding to $\theta$'s from non-conjugate chambers should be different
(as schemes over $X_0$) but we do not check that either (also this should follow from the classification of
the Procesi bundles).

\subsection{Main results}
 We write $X^\theta$ for the resolution of $X_0$ obtained by Hamiltonian reduction
with stability condition $\theta$. 

The main result of this paper is as follows.

\begin{Thm}\label{Thm:main1}
There are precisely $2|W|$ non-equivalent normalized Procesi bundles on $X^\theta$.
\end{Thm}

In fact, this theorem remains the same even if we do not consider $\K^\times$-equivariant structures, see Remark \ref{Rem:equi}
below.

\begin{Thm}\label{Thm:main2}
For any generic $\theta$ there is a normalized Procesi bundle $\Pro$ on $X^{\theta}$
depending on the chamber of $\theta$ with a property that $\Pro^{\Gamma_{n-1}}=\mathcal{T}^\theta$.
\end{Thm}

Let us explain the main ideas of the  proof.
We write $V^{reg}$ for the open subset of $V$
consisting of all points with trivial stabilizer in $\Gamma_n$. Then set $X_0^{reg}:=V^{reg}/\Gamma_n$
and $X^{reg}:=\pi^{-1}(X_0^{reg})$. Clearly, $\pi: X^{reg}\rightarrow X_0^{reg}$ is an isomorphism.
Now let $\Pro$ be a Procesi bundle. We claim that $\Pro|_{X^{reg}}=\eta_*(\Str_{V^{reg}})$.
Indeed,  \begin{align*} H^0(X,\Pro)=H^0(X, \mathcal{E}nd(\Pro)e)=H^0(X,\mathcal{E}nd(\Pro))e=(\K[V]\#\Gamma_n)e=\K[V],\end{align*} where $e:=\frac{1}{|\Gamma_n|}\sum_{\gamma\in \Gamma_n}\gamma$, and then we can just restrict to $X^{reg}$.
The reason for the existence of different Procesi
bundles is that $\operatorname{codim}_{X} X\setminus X^{reg}=1$. We are going to partially fix this by considering certain deformations
$\tilde{X},\tilde{X}_0$ of $X,X_0$, respectively, over the affine line $\mathbb{A}^1$.
These deformations will also be obtained by Hamiltonian reduction. An important property is that
the corresponding resolution of singularities $\tilde{\pi}:\tilde{X}\rightarrow \tilde{X}_0$
will be an isomorphism outside $\tilde{X}\setminus \tilde{X}^{reg}$ that now has codimension $2$. Since the higher
self-extensions of any $\mathcal{P}$ vanish and thanks to the $\K^\times$-equivariance,
the bundle $\Pro$ extends to a vector bundle
$\tilde{\Pro}$ on $\tilde{X}$. We will see that there are no more than $2|W|$ possibilities for the $\K[\tilde{X}_0]$-module
$H^0(\tilde{X},\tilde{\Pro})$. This was basically established in \cite{quant}. Since
$\tilde{\pi}$ is an isomorphism outside of a subvariety of codimension $2$, we see that if $H^0(\tilde{X},\tilde{\Pro}_1)\cong
H^0(\tilde{X},\tilde{\Pro}_2)$, then $\tilde{\Pro}_1\cong \tilde{\Pro}_2$. On the other hand, as we will see,
the construction in \cite{BK} yields $2|W|$ non-isomorphic Procesi bundles. This basically completes the
proof of Theorem \ref{Thm:main1}. In order to prove Theorem \ref{Thm:main2},
we will use similar ideas together with an argument due to Etingof and
Ginzburg, \cite{EG},  that shows that a general fiber of $\tilde{X}$ is the spectrum of the spherical
subalgebra in a suitable symplectic reflection algebra.

\subsection{Content  of the paper}
We prove Theorem \ref{Thm:main1} in Sections \ref{S_Pro_SRA},\ref{S_Pro_constr}.
In Section \ref{S_Pro_SRA} we prove that the number of different (normalized)
Procesi bundles does not exceed $2|W|$. The proof is based on a connection
between the Procesi bundles and symplectic reflection algebras (SRA) observed in
\cite{quant} and the deformation idea explained above. In Section \ref{S_Pro_constr}
we construct $2|W|$ different Procesi bundles, the construction is based on an
analysis of the approach to Procesi bundles from \cite{BK}. In Section
\ref{S_Pro_prop} we investigate properties of Procesi bundles, in particular,
proving Theorem \ref{Thm:main2}. We also describe the behavior of Procesi bundles
under parabolic restrictions. This will be used in our subsequent paper.

{\bf Acknowledgements}. I would have gotten nowhere in this project without numerous
conversations  with R. Bezrukavnikov. I would like to thank him and  M. Finkelberg, I. Gordon and M. Haiman
for stimulating discussions. Also I  am grateful to I. Gordon for remarks on a previous version of this text. This work
was supported by the NSF under Grants DMS-0900907, DMS-1161584.

\section{Procesi bundles vs symplectic reflection algebras}\label{S_Pro_SRA}
\subsection{Equivalence of normalized Procesi bundles}
Recall that we write $\Gamma_n$ for $\mathfrak{S}_n\ltimes \Gamma_1^n$ and $V$ for $\K^{2n}$.
Further, $X_0$ denotes the quotient $V/\Gamma_n$ and $X^\theta$ is its $\K^\times$-equivariant
symplectic resolution obtained by Hamiltonian reduction.

The goal of this subsection is to prove that, for  normalized Procesi bundles, a $\K^\times$-equivariant isomorphism
is the same as an equivalence. In other words, if $\Pro$ is a normalized Procesi bundle, then there is a unique
$\K^\times$-equivariant isomorphism $\End(\Pro)\xrightarrow{\sim} \K[\K^{2n}]\#\Gamma_n$ (modulo inner automorphisms). We write $V$ for
$\K^{2n}$ and identify $V$ with $V^*$ by means of the symplectic form. This allows us to identify $\K[V]=S(V^*)$ with $S(V)$.

\begin{Lem}\label{Lem:wreath_aut}
The group of outer graded automorphisms of the $S(V)^{\Gamma_n}$-algebra
$S(V)\#\Gamma_n$ is isomorphic to the character
group of $\Gamma_n$.
\end{Lem}
\begin{proof}
First of all, given a character $\chi:\Gamma_n\rightarrow \K^\times$, let us construct an automorphism
$a_\chi$ of $S(V)\#\Gamma_n$. Namely, we set $a_\chi(v)=v, a_\chi(\gamma)=\chi(\gamma)\gamma$.
This extends to a unique (automatically graded) $S(V)^{\Gamma_n}$-linear automorphism of $S(V)\#\Gamma_n$.
To see this we note that $S(V)\#\Gamma_n$ is the quotient of the coproduct $S(V)*\K\Gamma_n$ by the
relations $\gamma v \gamma^{-1}=\gamma(v)$ that are preserved by $a_\chi$.

The assignment $\chi\mapsto a_\chi$ defines an injective map from $\Hom(\Gamma_n, \K^\times)$
to the group $\operatorname{Out}$ of outer graded $S(V)^{\Gamma_n}$-algebra automorphisms of $S(V)\#\Gamma_n$.

Now let $a$ be some graded $S(V)^{\Gamma_n}$-linear automorphism of $S(V)\#\Gamma_n$.
For $x\in V^{reg}/\Gamma_n$, the fiber $(S(V)\#\Gamma_n)_x$ is a matrix algebra. So there is an open $\K^\times$-stable affine
covering  $V^{reg}/\Gamma_n=\bigcup_i U_i$ such that the automorphism $a_\chi$ is inner on $(S(V)\#\Gamma_n)_{U_i}$,
say,  given by an invertible element $a_i$.
This defines a 1-cocycle $a_{ij}:=a_i a_j^{-1}$ on $V^{reg}/\Gamma_n$ with coefficients in $\K^\times$, in other words, a line bundle
on $V^{reg}/\Gamma_n$. So we have a map $\operatorname{Out}\rightarrow \operatorname{Pic}(V^{reg}/\Gamma_n)$ which is
a group homomorphism, by the construction. This homomorphism is injective. Indeed, if $a$ lies in the kernel,
then we can choose the elements $a_i$ that agree on the intersections. But since $\operatorname{codim}_{V}V\setminus V^{reg} \geqslant 2$,
we see that $a_i$ glue to an element of $S(V)\#\Gamma_n$ and so $a$ is inner.

It remains to show that $\operatorname{Pic}(V^{reg}/\Gamma_n)$ equals $\Hom(\Gamma_n, \K^\times)$.
This follows from the observation that, since $\operatorname{codim}_{V}V\setminus V^{reg} \geqslant 2$, the group
$\operatorname{Pic}(V^{reg})$ is trivial.
\end{proof}

So if there are isomorphisms $\End(\Pro)\xrightarrow{\sim} S(V)\#\Gamma_n$ (satisfying $\Pro^{\Gamma_n}=\Str_X$)
that differ by an outer automorphism, then there is a non-trivial one-dimensional representation of $\Gamma_n$
such that the corresponding component of $\Pro$ is the (equivariantly) trivial line bundle, let $e'$ be the
corresponding idempotent. We have an isomorphism $(e+e')\End(\Pro)(e+e')\cong (e+e')(S(V)\#\Gamma_n)(e+e')$ of graded algebras.
The left hand side is $\End((e+e')\Pro)\cong \operatorname{Mat}_2(S(V)^{\Gamma_n})$. The component of degree
$0$ is therefore 4-dimensional. However, the component of degree $0$ in $(e+e')(S(V)\#\Gamma_n)(e+e')$
is 2-dimensional. This contradiction completes the proof of the claim in the beginning of the subsection.

\subsection{Universality property for SRA}
Let  $\param_{\Halg}$ be a vector space with basis $c_i$, where $i$ runs over the index set for the conjugacy classes
of symplectic reflections in $\Gamma_n$, and $h$. For $n>1$, we have the following classes
of symplectic reflections: $S_0$ containing all transpositions and $S_1,\ldots,S_r$, containing elements from non-unit conjugacy classes
$S_1^0,\ldots,S_r^0$ in the $n$ copies of $\Gamma_1$ inside  $\Gamma_n$. Below we often write $k$ for $c_0$. When $n=1$,
the class $S_0$ is absent.

Following  \cite{EG}, define an algebra $\Halg$ as the quotient of $S(\param_\Halg)\otimes T(V)\#\Gamma$
by the relations
\begin{align}\label{SRA_rel1}
&[x_{(i)},y_{(j)}]=-\frac{k}{2}\sum_{\gamma\in \Gamma_1}\omega_1(\gamma x,y)s_{ij}\gamma_{(i)}\gamma_{(j)}^{-1},\\
&[x_{(i)},y_{(i)}]=h\omega_1(x,y)+\frac{k}{2}\sum_{j\neq i}\omega_1(x,y)s_{ij}\gamma_{(i)}\gamma_{(j)}^{-1}+
\sum_{\ell=1}^r c_\ell\omega_1(x,y)\sum_{\gamma\in S^0_\ell}\gamma_{(i)}.
\end{align}
Here the following notation is used. For $\gamma\in \Gamma_1$ we write $\gamma_{(i)}$ for the element $\gamma$
in the $i$th copy of $\Gamma_1$. For $x,y\in \K^2$ the notation $x_{(i)},y_{(j)}$ has a similar meaning,
where we view $V$ as $(\K^2)^{\oplus n}$. We write $\omega_1$ for the symplectic form on $\K^2$
so that the form on $V$ is $\omega_1^{\oplus n}$.

The algebra $\Halg$ is a graded flat deformation of $S(V)\#\Gamma_n$ over $S(\param_{\Halg})$.

Consider the Hochschild cohomology $\HH^i(S(V)\#\Gamma_n), i\geqslant 0$. Each of these cohomology groups
is graded, the grading is induced by that on $S(V)\#\Gamma_n$. The  group $\HH^i(S(V)\#\Gamma_n)$ vanishes
in degrees less than $-i$. This follows, for example, from \cite[Proposition 6.2]{GK}, note that the isomorphism
of that proposition preserves the  gradings. So $S(V)\#\Gamma_n$ admits a universal graded deformation 
$\Halg_{univ}$ over the algebra $S(\param_{univ})$, where $\param_{univ}$  is the dual of the degree $-2$ 
component in $\HH^2(S(V)\#\Gamma)$. Here we view $\param_{univ}$ as the space concentrated in degree $2$. 
The universality is understood as follows: for any other
space $\param'$ concentrated in degree $2$ and any other graded flat deformation $\Halg'$ of $S(V)\#\Gamma$
over $S(\param')$ there is a unique linear map $\nu':\param_{univ}\rightarrow \param'$ and a unique isomorphism
$S(\param')\otimes_{S(\param_{univ})}\Halg_{univ}\xrightarrow{\sim} \Halg'$ of graded $S(\param')$-algebras
that is the identity modulo $\param'$, where the homomorphism
$S(\param_{univ})\rightarrow S(\param')$ is given by $\nu'$.

If $\Gamma_n\neq \mathfrak{S}_n$, then  the module $V$ is symplectically irreducible over $\Gamma_n$.
It follows that $\param_{univ}\cong \param_{\Halg}$. So $\Halg$ is a universal graded deformation of
$S(V)\#\Gamma$. We will assume that $\Gamma_1\neq \{1\}$ (this case will be considered in Subsection \ref{SS_symp_red}).
Also we will assume that $n>1$. The case $n=1$ is similar to (and simpler than) that one.

Here is a remark that will be of importance in the sequel. We can define $\Halg$ over a suitable
algebraic extension $R$ of $\Z$, let $\Halg(R)$ denote the corresponding $R$-algebra.
Enlarging $R$ if necessary, we can assume that the results on the
Hochschild cohomology quoted above hold for $\Halg(R)$ and so also for $\Halg(\Fi)$ for an arbitrary
$R$-algebra $\Fi$. In particular, $\Halg(\Fi)$ is again a universal graded deformation of $S_\Fi(V_\Fi)\#\Gamma_n$.

\subsection{Deformation of Procesi bundles and map $\nu_{\Pro}$}\label{SS_def_Pro}
Set $X:=X^\theta$.
Let us produce a $\K^\times$-equivariant deformation $\Dcal$ of $X$ obtained by quantum Hamiltonian reduction.

Consider the completed  homogenized
Weyl algebra $\W^{\wedge_h}_h(V)=T(V)[[h]]/ (u\otimes v-v\otimes u-h\omega(u,v))$. We have the quantum comoment map
$\Phi:\g\rightarrow \W^{\wedge_h}_h(V)$ induced by the natural inclusion of $\mathfrak{sp}(V)$ into the second graded component
of $\W_h(V)$. We can sheafify $\W^{\wedge_h}_h(V)$ to $V$  and get the $\K^\times$-equivariant sheaf $\W_{h,V}$. We then restrict
$\W_{h,V}/\W_{h,V}\Phi([\g,\g])$  sheaf-theoretically to $\mu^{-1}(\z^*)^{ss}$, and consider the sheaf $\pi_{G*}(\W_{h,V}/\W_{h,V}\Phi([\g,\g]))^G$ on $\mu^{-1}(\z^*)^{ss}\quo G$. Here $\pi_G$ is the quotient morphism.
Set $\param_{red}:=\z\oplus \K h$, where $\z:=\g/[\g,\g]= \K^{Q_0}$. The previous sheaf is that of $S(\param_{red})$-algebras.
Then we restrict this sheaf to $X^\theta$ and complete the restriction with respect to $\param_{red}$-adic
topology. The resulting sheaf is what we denote by $\Dcal$.

The sheaf $\Dcal$ is a sheaf of $\K[[\param_{red}^*]]$-algebras, flat over $\K[[\param_{red}^*]]$, complete and separated in the $\param_{red}$-adic topology and deforming $\Str_{X}$ in the sense that $\Dcal/\param_{red}\Dcal\cong \Str_{X}$.
The deformation $\Dcal/h\Dcal$ is commutative and is a Poisson sheaf of algebras deforming the Poisson sheaf
$\Str_X$. Here we consider the Poisson bracket on $\Dcal/h\Dcal$  induced by the bracket $\frac{1}{h}[\cdot,\cdot]$ on $\Dcal$.
We write $\Drm$ for the subalgebra in $H^0(X,\mathcal{D})$  of $\K^\times$-finite elements (=sums of $\K^\times$-semiinvariants).
The algebra $\Drm$ can be obtained via Hamiltonian reduction on the level of algebras:
$\Drm=[W_h(V)/W_h(V)\Phi([\g,\g])]^G$, see, for example, \cite[Lemma 4.2.4, Section 4.2]{quant}.

Now take a normalized Procesi bundle $\Pro$ on $X$. Thanks to (P2), the bundle $\Pro$ extends to a unique
$\K^\times$-equivariant sheaf $\tilde{\Pro}_h$ of right $\Dcal$-modules and the algebra
$\End_{\Dcal^{opp}}(\tilde{\Pro}_h)$ is a complete $\K^{\times}$-equivariant deformation of $\End_{\Str_X}(\Pro)$
over $\K[[\param_{red}^*]]$. Let $H_{\Pro}$ be the subalgebra of $\K^\times$-finite
elements of $\End_{\Dcal^{opp}}(\tilde{\Pro}_h)$. Then $H_{\Pro}$ is a graded deformation of $\End_{\Str_X}(\Pro)$ over $S(\param_{red})$.
It follows that there is a unique linear map $\nu_{\Pro}:\param_{univ}\rightarrow \param_{red}$ such that $H_{\Pro}=S(\param_{red})\otimes_{S(\param_{univ})}\Halg$.

We want to get some restrictions on the map $\nu_{\Pro}$. For this, let $e$ be the trivial idempotent in
$\K\Gamma_n$. The algebra $e H_{\Pro}e$ is naturally identified with $\Drm$. Indeed, recall that $e\Pro=\Str_X$.
Both $\Dcal$ and $e\tilde{\Pro}_h$ are deformations of $\Str_X$. Thanks to the cohomology vanishing for $\Str_X$, these deformations
coincide. On the other hand, $e \End_{\Dcal^{opp}}(\tilde{\Pro}_h) e=\End_{\Dcal^{opp}}(e\tilde{\Pro}_h)=H^0(X, \Dcal)$
and so $e H_{\Pro}e=\Drm$. So we get $\Drm=S(\param_{red})\otimes_{S(\param_{univ})}e\Halg e$. The group $W\times \Z/2\Z$
acts on $\param_{red}$ preserving $h$ and acting on $\z$ as in \cite[6.4]{quant}. The following lemma is slightly rephrased \cite[Proposition 6.4.5]{quant}.

\begin{Lem}\label{Lem:param_map}
Let $\nu: \param_{univ}\rightarrow \param_{red}$ be a linear map such that $\Drm\cong S(\param_{univ})\otimes_{S(\param_{red})}eHe$
(a $S(\param_{univ})$-linear isomorphism of graded algebras that is the identity modulo $\param_{univ}$). Then $\nu$
is an isomorphism. Moreover, any two different $\nu$ are obtained from one another by multiplying from the left by an element of
$W\times \Z/2\Z$.
\end{Lem}

One possible map $\nu$ can be described as follows. Let $\epsilon_i$ be the element in $\z$
corresponding to the vertex $i$ in the quiver. Then we can view $\epsilon_i,i=0,\ldots,r,$ and $h$
as a basis in $\param_{red}$. Also we have the basis $c_1,\ldots,c_r, k, h$  in $\param_{\Halg}=\param_{univ}$.
Form the element ${\bf c}\in \param_{univ}\otimes \K\Gamma_1$ by ${\bf c}=h+\sum_{i=1}^r c_i\sum_{\gamma\in S^0_i}\gamma$.
If $N_i$ is the representation of $\Gamma_1$ corresponding to $i$, then we can consider the element
$\operatorname{tr}_{N_i}{\bf c}\in \param_{univ}$.
Then for $\nu$ we can take the inverse of the following map $\param_{red}\xrightarrow{\sim}\param_{univ}$
\begin{equation}\label{eq:map} h\mapsto h, \epsilon_i\mapsto \operatorname{tr}_{N_i}{\bf c}, i\neq 0, \epsilon_0\mapsto \operatorname{tr}_{N_0}{\bf c}-(k+h)/2.
\end{equation}
\subsection{$\Pro$ is uniquely recovered from $\nu_\Pro$}
\begin{Prop}\label{Prop:Pro_coinc}
Let $\Pro^1,\Pro^2$ be normalized Procesi bundles on $X$ with $\nu_{\Pro^1}=\nu_{\Pro^2}$. Then $\Pro^1$
and $\Pro^2$ are equivalent.
\end{Prop}
\begin{proof}
Pick a generic element $\alpha\in \z^*$ and consider the corresponding one-parametric deformation
$\tilde{X}$ over $\K$. Let $\Pro^1_\alpha,\Pro_\alpha^2$ be the corresponding deformations of $\Pro^1,\Pro^2$
on $\tilde{X}$. We claim that $\Gamma(\Pro^1_\alpha)\cong \Gamma(\Pro^2_\alpha)$
(a $\K^\times$-equivariant isomorphism of $\K[\tilde{X}]$-modules). Indeed, we have $\Gamma(\tilde{\Pro}^i_h)_{fin}=\Halg e$
(a $\K^\times$-equivariant isomorphism of right $e\Halg e$-modules) and $\Gamma(\Pro^i_\alpha)$ is a specialization of
$\Gamma(\tilde{\Pro}^i_h)$ corresponding to the projection $\K[\z^*][h]\twoheadrightarrow \K[\K\alpha]$.
The specializations are the same for both $i=1,2$ exactly because $\nu_{\Pro^1}=\nu_{\Pro^2}$.
Consider the natural morphism $\tilde{\pi}:\tilde{X}\rightarrow \tilde{X}_0$, where $\tilde{X}_0$ is the spectrum
of $\K[\tilde{X}]$. Over the nonzero points of $\K$, the fibers of $\tilde{X}\rightarrow \K$ are affine, because $\alpha$ is generic.
It follows that $\tilde{\pi}$ is an isomorphism over $\K\setminus \{0\}$.  Over the zero point, $\tilde{\pi}$ is $\pi$ that is an isomorphism
generically. So $\tilde{\pi}$ is an isomorphism on the complement of $X\setminus X^{reg}$, a closed subvariety of
codimension $2$ in $\tilde{X}$. So $\tilde{\Pro}_1,\tilde{\Pro}_2$ are $\K^\times$-equivariantly isomorphic
bundles on the complement of $X\setminus X^{reg}$. It follows that $\tilde{\Pro}_1\cong\tilde{\Pro}_2$
and hence $\Pro_1\cong \Pro_2$. By the construction, this isomorphism induces the identity automorphism
of $\K[\K^{2n}]\#\Gamma_n$.
\end{proof}

For the future use, we remark that the argument makes sense over an algebra $\Fi$ over a suitable algebraic
extension $R$ of $\mathbb{Z}$.

\subsection{Symplectically reducible case}\label{SS_symp_red}
Now let us consider the group $\Gamma\subset \operatorname{Sp}(V)$ that is a product of several groups
of the form $\Gamma^i_{n_i}$, where $\Gamma^i_1\subset \SL_2(\K)$.
One way to get such a group is to take the stabilizer of some $v\in V$ inside
$\Gamma_n$. We can decompose $\Gamma$ as $\Gamma=\prod_{i=1}^j \Gamma^i$ and $V$ as $V=\bigoplus_{i=1}^j V^i\oplus V^{\Gamma}$,
where $V^i$ is an irreducible $\Gamma^i$-module.
Then, if we have a $\K^\times$-equivariant symplectic resolution $X\twoheadrightarrow V/\Gamma$ (for example,
the product of resolutions of individual quotients $V^i/\Gamma^i$), we can introduce the notion of
a Procesi bundle exactly as before. We claim, for the product resolution $X$ (where the factors are obtained
by Hamiltonian reduction), that modulo
Theorem \ref{Thm:main1}, any Procesi bundle on $X$ is obtained as a product of Procesi bundles on the factors
(clearly any Procesi bundle on $V^{\Gamma}$ viewed as a resolution of itself is trivial).

First of all, let us remark that we still have a universal graded deformation $\Halg_{univ}$ of $S(V)\#\Gamma$. It is obtained
as the product of the SRA's $\Halg(V^i,\Gamma^i)$ (with parameter spaces $\param_{\Halg}^i$)
and the universal Weyl algebra (over $\bigwedge^2 V^\Gamma$) of the space $V^{\Gamma}$ (and so the space
of parameters $\param_{univ}$ for the universal deformation is $\bigoplus_{i=1}^j \param^i_{\Halg}\oplus \bigwedge^2 V^{\Gamma}$).
We still can consider the symplectic reflection algebra $\Halg$ for $(V,\Gamma)$ that is a specialization
of the universal deformation. It can be characterized as a unique specialization with the following property:
any one-parametric deformation $\Halg'$ of $S(V)\#\Gamma$ such that the Poisson bracket on $S(V)^{\Gamma}$
induced by $e\Halg' e$ is proportional to the standard bracket factors through $\Halg$.

Also we can consider the deformations $\mathcal{D}$ of $X$ and $\Drm$ of $V/\Gamma$ obtained by Hamiltonian
reduction. The deformation $\Drm$ still has the compatibility with the Poisson bracket mentioned in
the previous paragraph. It follows that the deformation $\Halg_{\Pro}$ factors through $\Halg$.
Lemma \ref{Lem:param_map} generalizes to this situation in a straightforward way, the corresponding
automorphism group is the product of those for individual factors. From here we deduce that the map
$\nu_{\Pro}$ is basically the product of the maps $\nu_{\Pro^i}$ (more precisely, this is true as stated
on the affine hyperplane given by $h=1$). Proposition \ref{Prop:Pro_coinc} generalizes verbatim together
with its proof. The claim in the beginning of the subsection follows.

\section{Construction of Procesi bundles}\label{S_Pro_constr}
The main goal of this section is to produce $2|W|$ non-isomorphic Procesi bundles on $X=X^\theta$ following Bezrukavnikov and Kaledin.
This will complete the proof of part (1) of the main theorem.

\subsection{From Azumaya algebras to Procesi bundles}\label{SS_Az_to_Pro}
In this subsection we recall a construction of a Procesi bundle on $X$ given in \cite{BK}.
The base field in their construction is an algebraically closed field $\Fi$ of characteristic
$p\gg 0$. But all objects we will consider are actually defined over a finite subfield
of $\Fi$. We now set $V=\Fi^{2n}$ and let $X$ be a resolution of $V/\Gamma_n$ obtained by
Hamiltonian reduction. As in \cite{BK}, the superscript ``$(1)$'' means the Frobenius twist.

An input for the construction is a microlocal quantization $\mathrm{O}$ of $X$ that has the following properties:
\begin{enumerate}
\item The center of $\mathrm{O}$ is $\operatorname{Fr}^*\mathcal{O}_{X^{(1)}}$, where $\operatorname{Fr}:X\rightarrow X^{(1)}$ is a Frobenius morphism.
\item We have a $S(V^{(1)})^{\Gamma_n}$-algebra isomorphism $\Gamma(\mathrm{O})\cong \Aa^{\Gamma_n}$, where
$\Aa$ is the Weyl algebra of $V$.
\end{enumerate}

Let us describe how to recover a Procesi bundle $\Pro$ (on $X^{(1)}$; this variety is equivariantly isomorphic to $X$)
from $\mathrm{O}$. The push-forward $\operatorname{Fr}_*\mathrm{O}$ is an Azumaya algebra over $X^{(1)}$ (below we abuse the notation
and write $\mathrm{O}$ instead of $\operatorname{Fr}_*\mathrm{O}$  meaning an Azumaya algebra).
Let $\hat{X}^{(1)}$ denote the formal neighborhood of the preimage of $0$ (under $\pi$) in $X^{(1)}$.
Then it follows from \cite[6.2]{BK} that $\mathrm{O}|_{\hat{X}^{(1)}}$ splits. Fix a splitting bundle $\mathcal{S}$
(recall that it is defined up to a twist with a line bundle).
Let us produce a bundle $\hat{\Pro}_{\mathcal{S}}$ on $\hat{X}^{(1)}$. Namely, we have an equivalence
$\operatorname{Coh}(\hat{X}^{(1)})\xrightarrow{\sim} \operatorname{Coh}(\hat{X}^{(1)},\mathrm{O})$
given by $\mathcal{S}\otimes\bullet$. Then the derived global sections functor $R\Gamma$ gives
an equivalence $D^b(\operatorname{Coh}(\hat{X}^{(1)},\mathrm{O}))\xrightarrow{\sim} D^b(\hat{\W}^{\Gamma_n}\operatorname{-mod})$,
where  $\hat{\W}$ stands for the  completion of the Azumaya algebra $\W$ with respect to $0\in V^{(1)}$.
Next, we have a Morita equivalence between $\hat{\W}^{\Gamma_n}$ and $\hat{\W}\#\Gamma_n$:
$\hat{\W}^{\Gamma_n}\operatorname{-mod}\xrightarrow{\sim}\hat{\W}\#\Gamma_n\operatorname{-mod}$.
The Azumaya algebra $\hat{\W}$ on $V^{(1)}$ admits a $\Gamma$-equivariant splitting bundle, say $\mathcal{Q}$,
this is shown using an argument of the proof of \cite[Theorem 6.7]{BK}. We fix this splitting once and for all.
It gives rise to an equivalence $\hat{\W}\#\Gamma_n\operatorname{-mod}\xrightarrow{\sim}
\operatorname{Coh}^{\Gamma_n}(\hat{V}^{(1)})$, where $\hat{V}^{(1)}$ denotes the formal
neighborhood of $0$ in $V^{(1)}$. Let $\iota_{\mathcal{S}}:D^b(\operatorname{Coh}(X^{(1)}))\xrightarrow{\sim} D^b(\operatorname{Coh}^{\Gamma_n}(\hat{V}^{(1)}))$ be the resulting equivalence. Then we set $\hat{\Pro}_{\mathcal{S}}:=\iota_{\mathcal{S}}^{-1}(\mathcal{O}_{\hat{V}^{(1)}}\#\Gamma_n)$. This is a
vector bundle on $\hat{X}^{(1)}$, this is proved completely analogously to a similar argument
in the proof of \cite[Theorem 6.7]{BK}. We remark that, by the construction, $\hat{\Pro}_{\mathcal{S}}$
has vanishing higher Ext's and comes equipped with an isomorphism
$\operatorname{End}(\hat{\Pro}_S)\xrightarrow{\sim}\Fi[[V^{(1)}]]\#\Gamma_n$. Therefore we can normalize
$\mathcal{S}$ and  $\hat{\Pro}_{\mathcal{S}}$ by requiring that $\hat{\Pro}_{\mathcal{S}}^{\Gamma_n}$
is the structure sheaf. The corresponding bundle will be denoted by $\hat{\Pro}_{\mathrm{O}}$.

Let us show how to recover an $\Fi^\times$-equivariant structure on $\hat{\Pro}_{\mathrm{O}}$.
For an irreducible $\Gamma_n$-module $N$,  let $\hat{\Pro}_{\mathrm{O}}^N:=\Hom_\Gamma(N, \hat{\Pro}_{\mathrm{O}})$
so that $\hat{\Pro}_{\mathrm{O}}=\bigoplus_N \hat{\Pro}_{\mathrm{O}}^N\otimes N$.
The bundles $\hat{\Pro}_{\mathrm{O}}^N$ are indecomposable.
To see  this one can argue as follows. The algebra
$\operatorname{End}(\hat{\Pro}_{\mathrm{O}}^N)$ coincides with $e_N \operatorname{End}(\hat{\Pro}_{\mathrm{O}})e_N=
e_N(\Fi[[V^{(1)}]]\#\Gamma_n)e_N$, where $e_N\in \Fi\Gamma$ is an indecomposable idempotent corresponding to $N$.
The algebra $e_N(\Fi[[V^{(1)}]]\#\Gamma_n)e_N$ is complete in the $V^{(1)}$-adic topology and the quotient modulo
$V^{(1)}$ is $\Fi$. So if $e',e''\in e_N(\Fi[[V^{(1)}]]\#\Gamma_n)e_N$ are two commuting idempotents with $e'e''=0$,
then one of $e',e''$ is zero. This shows that $\hat{\Pro}_{\mathrm{O}}^N$ is indecomposable.

So the bundles $\hat{\Pro}_{\mathrm{O}}^N$ are indecomposable and have trivial higher Ext's.
So, according to \cite{Vologodsky}, each admits an equivariant structure, unique up to a twist with a character of $\Fi^\times$.

\begin{Lem}
 After suitable twists, there is an $\Fi^\times$-equivariant $\Fi[[V^{(1)}]]^{\Gamma_n}$-equivariant isomorphism $\End(\hat{\Pro}_{\mathrm{O}})\xrightarrow{\sim} \Fi[[V^{(1)}]]\#\Gamma_n$.
\end{Lem}
\begin{proof}
We have some $\Fi[[V^{(1)}]]^{\Gamma_n}$-equivariant isomorphism by the construction of $\hat{\Pro}_{\mathrm{O}}$. First, we will show that
there are twists with the $\Fi[[V^{(1)}]]^{\Gamma_n}$-modules $\Fi[[V^{(1)}]],\Gamma(\Pro_{\mathrm{O}})$ being $\Fi^\times$-equivariantly isomorphic. For this it is enough to show that the $\Fi[[V^{(1)}]]^{\Gamma_n}$-modules $\Fi[[V^{(1)}]]^N:=\Hom_{\Gamma_n}(N, \Fi[[V^{(1)}]])$, where $N$ is an irreducible $\Gamma_n$-representation, are indecomposable and each  admits a unique $\Fi^\times$-equivariant structure up to a twist.

We start with some notation.  Let $V^{(1) reg}$ denote the open subvariety of $V^{(1)}$ consisting of points with free
$\Gamma_n$-orbits. Let $\hat{V}^{(1)}$ denote the formal neighborhood  of $0$ in $V^{(1)}$.
We set $\hat{V}^{(1)reg}:=\hat{V}^{(1)}\cap V^{(1)reg}$, we would like to emphasize that $\hat{V}^{(1)reg}$
is a $\Fi$-scheme, and not a formal scheme. The restriction $\Fi[[V^{(1)}]]^N|_{\hat{V}^{(1)reg}/\Gamma_n}$
is a vector bundle. Since $\Fi[[V^{(1)}]]^N$ is a torsion-free $\Fi[[V^{(1)}]]^{\Gamma_n}$-module, our claim in
the end of the previous paragraph will follow if we show that $\Fi[[V^{(1)}]]^N|_{\hat{V}^{(1)reg}/\Gamma_n}$
is indecomposable and admits a unique $\Fi^\times$-equivariant structure up to a twist. In its turn, this will
follow if we check that the pull-back $\mathcal{O}_{\hat{V}^{(1)reg}}\otimes_{\Fi[[V^{(1)}]]^{\Gamma_n}}\Fi[[V^{(1)}]]^N$
is indecomposable as an $\Fi[[V^{(1)}]]\#\Gamma_n$-module and admits a unique $\Fi^\times$-equivariant structure
up to a twist.

But $\mathcal{O}_{\hat{V}^{(1)reg}}\otimes_{\Fi[[V^{(1)}]]^{\Gamma_n}}\Fi[[V^{(1)}]]^N=\mathcal{O}_{\hat{V}^{(1)reg}}\otimes_{\Fi} N$.
The global sections of the latter bundle is $\Fi[[V^{(1)}]]\otimes_{\Fi}N$, an indecomposable $\Fi[[V^{(1)}]]\#\Gamma_n$-module
that clearly has a unique $\Fi^\times$-equivariant structure up to a twist. The analogous claims about  $\mathcal{O}_{\hat{V}^{(1)reg}}\otimes_{\Fi[[V^{(1)}]]^{\Gamma_n}}\Fi[[V^{(1)}]]^N$ and $\Fi[[V^{(1)}]]^N|_{\hat{V}^{(1)reg}/\Gamma_n}$
follow. Since the latter is the restriction of $\Fi[[V^{(1)}]]^N$ to an open subscheme, we deduce the analogous claims
for $\Fi[[V^{(1)}]]^N$. This implies, that, after a twist, the $\Fi^\times$-equivariant $\Fi[[V^{(1)}]]^{\Gamma_n}$-modules
$\Fi[[V^{(1)}]]$ and $\Gamma(\hat{\Pro}_{\mathrm{O}})$ become isomorphic.

Then, if we endow $\Fi[[V^{(1)}]]\#\Gamma_n$ with  an $\Fi^\times$-action coming from $\operatorname{End}(\mathcal{P}_{\mathrm{O}})$, a natural homomorphism $\Fi[[V^{(1)}]]\#\Gamma_n\rightarrow \operatorname{End}_{\Fi[[V^{(1)}]]^{\Gamma_n}}(\Fi[[V^{(1)}]])$ becomes an $\Fi^\times$-equivariant homomorphism of $\Fi[[V^{(1)}]]^{\Gamma_n}$-algebras. But, as in the proof of \cite[Theorem 1.5]{EG}, the homomorphism is an isomorphism. This completes the proof of the lemma.
\end{proof}

It follows that we can extend $\hat{\Pro}_{\mathrm{O}}$ to a bundle $\Pro_{\mathrm{O}}$ on $X^{(1)}$
and that $\Pro_{\mathrm{O}}$ is a Procesi bundle.

\begin{Rem}\label{Rem:equi}
The same argument shows that we do not need to specify the $\K^\times$-equivariant structure in the definition
of the Procesi bundle: the bundle will automatically come with such a structure (normalized by the condition
that $\Pro^{\Gamma_n}=\Str_X$ as equivariant bundles).
\end{Rem}

The picture above will be useful to check that different quantizations $\mathrm{O}$ give different
normalized Procesi bundles. The first claim in this direction is below.

\begin{Lem}\label{Lem:Azumaya_equal}
Let $\Az^1,\Az^2$ be two Azumaya algebras on $X^{(1)}$ as before. Suppose that $\Pro_{\Az^1}$
is equivalent to $\Pro_{\Az^2}$. Then $\hat{\Az}^1\cong \hat{\Az}^2$.
\end{Lem}
\begin{proof}
Let $\mathcal{S}^1,\mathcal{S}^2$ be the corresponding splitting bundles. We are going to
prove that $\mathcal{S}^{1*}\cong \mathcal{S}^{2*}$. Namely, we notice that since
$\Pro_{\Az^1}\cong \Pro_{\Az^2}$, the equivalences $\kappa^i: D^b(\operatorname{Coh}(X^{(1)}))\xrightarrow{\sim}
D^b(\hat{\W}^{\Gamma_n})$ defined by $\mathcal{S}^i$ are isomorphic functors. We conclude that
$\kappa^1(\mathcal{S}^{j*})\cong \kappa^2(\mathcal{S}^{j*})$ for any $j=1,2$.
But, by the construction, $\kappa^j(\mathcal{S}^{j*})=R\Gamma(\mathcal{S}^j\otimes\mathcal{S}^{j*})=R\Gamma(\hat{\Az}^j)=
\hat{\W}^{\Gamma_n}$. It follows that $\kappa^1(\mathcal{S}^{2*})\cong \kappa^2(\mathcal{S}^{2*})=\hat{\W}^{\Gamma_n}=\kappa^1(\mathcal{S}^{1*})$. Since $\kappa^1$ is
an equivalence, we conclude that $\mathcal{S}^{1*}\cong \mathcal{S}^{2*}$.
\end{proof}

Also from the above construction of a Procesi bundle we see that the indecomposable components
of its restriction to $\hat{X}^{(1)}$ are (=the restrictions of the indecomposable components)
are the components of a splitting bundle for $\hat{\Az}$.

\subsection{Construction of $\Az$}
We are going to construct the Azumaya algebras $\Az$ by means of Hamiltonian reduction.
We need to produce $2|W|$ different Azumaya algebras (and formally even more, they need
to be non-isomorphic after the restriction to $\hat{X}^{(1)}$). We will produce the required number
of algebras in this subsection and show that they are different in the next one.

Recall that in \cite{quant} we established an isomorphism $e\Halg e\xrightarrow{\sim}\operatorname{D}$
and have produced an action of $W\times \Z/2\Z$ on $e \Halg e\cong \operatorname{D}$
by graded automorphisms that was the identity modulo the parameters. This was done over in arbitrary algebraically closed field
of characteristic $0$, in particular we can do this over $\overline{\mathbb{Q}}$. By the argument in
Subsection \ref{SS_def_Pro}, there is an isomorphism $e \Halg_R e\xrightarrow{\sim} \operatorname{D}_R$ (where $\operatorname{D}_R$
is defined by means of a quantum Hamiltonian reduction over $R$) produced by a Procesi bundle, for
a  sufficiently large algebraic extension $R$ of $\Z$. Extending $R$ further, we may assume that the
$W\times \Z/2\Z$-action is defined over $R$.  So we can assume that the isomorphism $e\Halg e\xrightarrow{\sim}
\operatorname{D}$ holds over $\Fi$ and we have an action of $W\times \Z/2\Z$ on these algebras by automorphisms.
Let us remark that the parameter $\lambda$ corresponding to $\Halg_\lambda=\W\#\Gamma_n$ (i.e. $k=c_i=0$)
has trivial stabilizer in $W\times \Z/2\Z$. This can be either deduced from an explicit description
of the action (the action on $\z$ is as described in Subsection \ref{SS_Ham_res} and $h$ is fixed)
or similarly to the proof of \cite[Proposition 6.4.5]{quant}.
Our conclusion is that we have $2|W|$ different parameters $\lambda$ such that $\operatorname{D}_\lambda\cong \W^{\Gamma_n}$
(that lie in $\Fi_p$ and are still different for $p\gg 0$).
Now let $\Az_\lambda$ denote the quantization of $X$ obtained by Hamiltonian reduction with parameter $\lambda$
(a microlocal sheaf of algebras). Since $\Gamma(X,\Az_\lambda)=\Drm_\lambda$, we are done.

\subsection{$\hat{\Az}_{\lambda'}\not\cong \hat{\Az}_\lambda$ for $\lambda\neq \lambda'$}
We are going to show that when $\lambda,\lambda'$ are different $W\times \Z/2\Z$- conjugate parameters,
then $\hat{\Az}_\lambda\not\cong \hat{\Az}_{\lambda'}$. This will complete the proof that $\Pro_{\Az_{\lambda}}$
and $\Pro_{\Az_{\lambda'}}$ are not equivalent therefore producing $2|W|$ different normalized Procesi bundles over  $\Fi$
(and defined over some finite subfield). This argument was suggested to us by R. Bezrukavnikov.

Recall that $\lambda-\lambda'$ is an element in $\Fi_p\otimes \mathfrak{X}(G)$, where $\mathfrak{X}(G)$
denotes the character lattice of $G$. Let us lift it to an element $\mu\in\mathfrak{X}(G)$.
Then we can produce a $\Az_{\lambda'}$-$\Az_{\lambda}$-bimodule $\Az_{\lambda,\mu}$ that quantizes the line bundle $\Str_\mu$
on $X$ and gives a Morita equivalence between the quantizations $\Az_{\lambda'},\Az_{\lambda}$.
This $\Az_{\lambda,\mu}$ (or, more precisely, $\operatorname{Fr}_* \Az_{\lambda,\mu}$) is a vector bundle on $X^{(1)}$
that is a splitting bundle for $\Az_{\lambda'}\otimes_{\Str_{X^{(1)}}}\Az_{\lambda}^{opp}$. Therefore
$\hat{\Az}_{\lambda,\mu}$ is a splitting bundle for  $\hat{\Az}_{\lambda'}\otimes_{\hat{\Str}_{X^{(1)}}}\hat{\Az}_{\lambda}^{opp}$.
It follows that $\hat{\Az}_{\lambda,\mu}$ is obtained from $\hat{\Az}_{\lambda}$ by a twist by a line bundle, say $\Str_{\mu'}$.
Let us show that this is impossible even on the level of $K_0$.

The space $K^{\K}_0(\operatorname{Coh}(\hat{X}^{(1)}))$ has a filtration by dimension of support.
The multiplication by a vector bundle of rank $r$ preserves this filtration and acts by $r$
on the successive quotients.
In particular, the multiplication by a line bundle is a unipotent operator.
So one can exponentiate the classes of line bundles to any power. Since $\hat{\Az}_{\lambda,\mu}$ is a deformation of $\Str_\mu$,
we see that   the class
$[\operatorname{Fr}_* \hat{\Az}_{\lambda,\mu}]$ equals $[\operatorname{Fr}_* \Str_\mu]$. The latter equals $[\Str_{\mu}]^{1/p}[\operatorname{Fr}_*\Str_{\hat{X}}]$. Since $[\operatorname{Fr}_*\Str_{\hat{X}}]$
is not a zero divisor ($\operatorname{Fr}_*\Str_{\hat{X}}$ is a vector bundle),
we see that $[\Str_{\mu}]^{1/p}=[\Str_{\mu'}]$ meaning that $\mu=p\mu'$. A contradiction
with $\lambda\neq \lambda'$ in $\Fi_p$.

\subsection{Lifting to characteristic $0$}
The last step in the construction of a Procesi bundle  in \cite{BK} was to lift the bundle
constructed in characteristic $p\gg 0$ to characteristic $0$. We only need to make sure that
under this construction non-isomorphic bundles remain non-isomorphic.

Pick a finite subfield $\Fi_q\subset \Fi$ such that all $2|W|$ Procesi bundles are defined
over $\Fi_q$. Further, enlarge $R$ and $q$ so that $\Fi_q$ becomes a residue field for $R$, let $\mathfrak{m}$ denote the corresponding
maximal ideal. Then thanks to the Ext vanishing, we can lift any Procesi bundle $\Pro$ to $R^{\wedge_{\mathfrak{m}}}$.
Then we can localize to the generic point of $R^{\wedge_{\mathfrak{m}}}$. Thanks to \cite[Lemma 6.9]{BK}, the resulting
bundle is defined over a finite extension of $\mathbb{Q}$.

To check that  non-isomorphic bundles $\Pro^1,\Pro^2$ remain non-isomorphic we notice that
that the maps $\nu_{\Pro^1},\nu_{\Pro^2}$ remain different under our transformations.
Indeed, when we lift $\Pro$ from $\Fi$ to $R^{\wedge_{\m}}$, the map $\nu_{\Pro}$ also
lifts from $\Fi$ to $R^{\wedge_{\m}}$. The behavior of this map under the other transformations
is clear.

Our conclusion is that there are $2|W|$ non-isomorphic normalized Procesi bundles over $\K$.
This completes the proof  of Theorem \ref{Thm:main1}.

\section{Properties of Procesi bundles}\label{S_Pro_prop}
The base field in what follows is $\K$. We consider a symplectic resolution $X=X^\theta$ of $X_0=V/\Gamma_n$.
\subsection{Restrictions of Procesi bundles}
We write $\Gamma$ for $\Gamma_n$.
Pick $b\in V$. We want to describe the restriction of a Procesi bundle $\Pro$
to a formal neighborhood $X^{\wedge_b}$ of $\pi^{-1}(\eta(b))$, where recall $\eta:V\rightarrow V/\Gamma$
denotes the quotient morphism. The structure of this formal neighborhood
can be described as follows. First of all, recall that $(V/\Gamma)^{\wedge_{\eta(b)}}$ is naturally
identified with $V^{\wedge_0}/\Gamma_b$. Therefore  $X^{\wedge_b}$
is a symplectic resolution  of $V^{\wedge_0}/\Gamma_b$. In fact, this resolution
can be obtained by Hamiltonian reduction. To $\eta(p)$ there corresponds a point  $x\in\mu^{-1}(0)$
with closed orbit. Consider the symplectic part $N_x=(T_xGx)^{\perp}/T_x Gx$ of the slice $G_x$-module.
The $G_x$-module $N_x$
still corresponds to a double framed quiver, let $\underline{\mu}$ denotes the corresponding
moment map. It is easy to see that $X^{\wedge_b}$ is the formal
neighborhood of the zero fiber in $\underline{X}:=\underline{\mu}^{-1}(0)^{\theta-ss}/G_x$, where we write $\theta$
for the restriction of the original character $\theta:G\rightarrow \K^\times$ to $G_x$.

Our main result in this subsection is as follows. We have a natural map from the space $\param_{\Halg,b}$ constructed
for $\Gamma_b$ to $\param_{\Halg}$ (the element corresponding to a conjugacy class in $\Gamma_b$ is sent
to the element corresponding to the conjugacy class in $\Gamma$ containing the class in $\Gamma_b$).
For example, if $\Gamma_b=\Gamma_{n-m}\times \mathfrak{S}_m$ (with $n-m,m>1$), then the space
$\param_{\Halg,b}$ is spanned by $h,c_1,\ldots,c_r, k,k'$ (where $k'$ corresponds to the reflections in $\mathfrak{S}_m$)
and the map sends $h$ to $h$, $c_i$ to $c_i$ and $k,k'$ to $k$.
Similarly, we have a natural map $\param_{red,b}=\g_x/[\g_x,\g_x]\oplus \K h\rightarrow \param_{red}=\g/[\g,\g]\oplus \K h$
induced by the inclusion $\g_x\hookrightarrow \g$.

\begin{Prop}\label{Prop:Proc_restr}
Under the identification $X^{\wedge_b}\cong \underline{X}^{\wedge_0}$, we have an isomorphism
of $\Pro^{\wedge_b}$ and $\Hom_{\K\Gamma_b}(\K\Gamma,\underline{\Pro}^{\wedge_0})$ of the restrictions
of bundles to formal neighborhoods. Here $\underline{\Pro}$ is a Procesi bundle on $\underline{X}$ that is unique
with the property that the following diagram commutes.
\begin{center}
\begin{picture}(60,30)
\put(2,2){$\param_{res,b}$}\put(40,2){$\param_{red}$}
\put(3,22){$\param_{\Halg,b}$}\put(41,22){$\param_{\Halg}$}
\put(5,21){\vector(0,-1){15}}\put(43,21){\vector(0,-1){15}}
\put(11,3){\vector(1,0){28}}\put(11,23){\vector(1,0){28}}
\end{picture}
\end{center}
\end{Prop}
\begin{proof}
The endomorphism algebra of $\Pro^{\wedge_b}$ is the completion $(SV\#\Gamma)^{\wedge_b}$ of $SV\#\Gamma$
at $\eta(b)$. This algebra is naturally isomorphic to the centralizer algebra $Z(\Gamma,\Gamma_b, SV^{\wedge_0}\#\Gamma_b)$,
see \cite[2.3]{sraco}. Here $Z(\Gamma,\Gamma_b, SV^{\wedge_0}\#\Gamma_b)$ is the endomorphism of the right
$SV^{\wedge_0}\#\Gamma_b$-module $\Hom_{\Gamma_b}(\Gamma, SV^{\wedge_0}\#\Gamma_b)$, see loc. cit. for details.
It follows that $\Pro^{\wedge_b}\cong \Hom_{\Gamma_b}(\Gamma, \Pro')$ for a bundle $\Pro'$ on $\underline{X}^{\wedge_0}$.
By the construction, we have $\End(\Pro')\xrightarrow{\sim} SV^{\wedge_0}\#\Gamma_b, \Ext^i(\Pro',\Pro')=0$
for $i>0$ and $\Pro'^{\Gamma_b}=\Str_{\underline{X}^{\wedge_0}}$. As above (Subsection \ref{SS_Az_to_Pro}),
we equip $\Pro'$ with a $\K^\times$-equivariant
structure and extend it to $\underline{X}$, getting a Procesi bundle, $\underline{\Pro}$.
It remains to check that the resulting map $\nu_{\underline{\Pro}}$ makes the diagram above commutative and that
$\underline{\Pro}$ is unique with this property.

Let us prove the first statement. We have an isomorphism $\Dcal^{\wedge_b}\cong \underline{\Dcal}^{\wedge_0}$, where $\underline{\Dcal}$
is an analog of $\Dcal$ for $\underline{X}$.
This can be established by analogy with \cite[Lemma 6.5.2]{quant} or from the observations that both sides
are canonical quantizations of $X^{\wedge_b}$ in the sense of Bezrukavnikov and Kaledin, \cite{BK1},
see also \cite[Section 5]{quant}. So the isomorphism $\Pro^{\wedge_b}\cong
\Hom_{\Gamma_b}(\Gamma,\underline{\Pro}^{\wedge_0})$ gives rise to an isomorphism
$\Halg^{\wedge_b}\xrightarrow{\sim} Z(\Gamma,\Gamma_b, \underline{\Halg}^{\wedge_0})$,
where the algebra $\underline{\Halg}^{\wedge_0}$ is defined over $S(\param_{\Halg})$. By the construction, this isomorphism
lifts the isomorphism $(SV\#\Gamma)^{\wedge_b}\xrightarrow{\sim} Z(\Gamma,\Gamma_b, SV^{\wedge_0}\#\Gamma_b)$.
There is such an isomorphism $\Halg^{\wedge_b}\xrightarrow{\sim} Z(\Gamma,\Gamma_b, \underline{\Halg}^{\wedge_0})$
that is $S(\param_{\Halg})$-linear, see \cite{sraco}. From the two isomorphisms,
we get an automorphism of $Z(\Gamma,\Gamma_b, \underline{\Halg}^{\wedge_0})$ that  is the identity on the quotient
$Z(\Gamma,\Gamma_b, SV^{\wedge_0}\#\Gamma_b)$. We can compose this automorphism with an inner automorphism to make it
$\K^\times$-equivariant because the eigenvalues of the Euler vector field
on the kernel of $Z(\Gamma,\Gamma_b, \underline{\Halg}^{\wedge_0})\twoheadrightarrow Z(\Gamma,\Gamma_b, SV^{\wedge_0}\#\Gamma_b)$
are all positive. It follows that this isomorphism preserves $\param_{\Halg}$.
Since the graded deformations of $SV\#\Gamma_b$ are unobstructed, the automorphism of $\param_{\Halg}$ has to
be the identity on the image of $\param_{\Halg,b}$. This proves our compatibility claim.

All possible maps $\nu_{\underline{\Pro}}$ are constructed as was explained in Subsection \ref{SS_symp_red}.
Analyzing the list, we easily see that only one map makes the diagram in the statement of the proposition
commutative. As we have seen in Subsection \ref{SS_symp_red}, the bundle $\underline{\Pro}$ is uniquely
recovered from $\nu_{\underline{\Pro}}$.
\end{proof}

\subsection{$\Gamma_{n-1}$-invariants}
We are going to prove that the subbundle of the $\Gamma_{n-1}$-invariants in a suitable Procesi bundle on $X^\theta$ 
coincides with the tautological bundle $\mathcal{T}$.  Namely, we will take the bundle $\Pro$ such that
$\nu_\Pro$ is the inverse of the map given by (\ref{eq:map}).

For $\theta$, we can have either $\theta\cdot \delta=\sum_{i\in Q_0}\theta_i\delta_i>0$ or $\theta\cdot\delta<0$. We consider the case when
$\theta\cdot\delta>0$ first. As was shown in \cite{GG}, the variety $\mu^{-1}(0)=\{(A,B,i,j)| [A,B]+ij=0\}$
has $n+1$ irreducible components. Let $\Lambda_+$ be the irreducible component, where $j=0$,
and $\Lambda_-$ be the irreducible component, where $i=0$.

\begin{Lem}
Let $\theta\cdot\delta>0$. Then $\mu^{-1}(0)^{ss}\subset \Lambda_+$.
\end{Lem}
\begin{proof}
Consider a point $r=\{(A,B,0,0)\}$ with generic commuting $A,B$. The corresponding representation of
the quiver $Q$ is semi-simple, in other words, $Gr$ is closed. The group $G_r$ is the center of $\GL(\delta)^{\times n}$,
i.e., the $n$-dimensional torus naturally identified with $(\K^\times)^n$. The restriction of $\theta$ to
all copies of $\K^\times$ is $\theta\cdot\delta$. The slice module for $r$ is the sum
of a trivial $G_r$-module and $\K^n\oplus \K^{n*}=\{(i,j)\}$. Since $\theta\cdot\delta>0$, the stable points
with respect to the induced stability condition are those, where all coordinates of $i$ are nonzero,
and, correspondingly, $j=0$. This proves our claim.
\end{proof}

Let $\alpha\in \z^*$ and the one-parameter deformation  $\tilde{X}$ of $X$ be as in the proof of Proposition \ref{Prop:Pro_coinc}.
Let $\tilde{\mathcal{T}}$ denote the tautological bundle on the resolution.
Now consider the variety $\tilde{X}^{reg}:=\tilde{X}\setminus (X\setminus X^{reg})$
and the bundle $\tilde{\Pro}^{\Gamma_{n-1}}$ on this variety.
The elements $x_n,y_n\in \Halg$ act on $\tilde{\Pro}$ fiberwise. Let us choose an identification in $\tilde{\Pro}^{\Gamma_{n-1}}_x
\cong (\K\Gamma_n)^{\Gamma_{n-1}}$ of modules over the centralizer of $\Gamma_{n-1}$ in $\K\Gamma_n$.
This presents $x_n,y_n$ as operators $A,B$ on $(\K\Gamma_n)^{\Gamma_{n-1}}$. For $i$ take the image
of $1$ in $\K\Gamma_n$. Then, as Etingof and Ginzburg essentially checked in \cite[Lemma 11.15]{EG},
for $x$ lying over a nonzero point $z\in\mathbb{A}^1$ there is a unique element $j\in \K^{n*}$
such that $(A,B,i,j)\in \mu^{-1}(z)$ (here we use the form of $\nu_\Pro$). If $x\in X^{reg}$, then it is easy
to see that $\K[x_n,y_n]1=(\K\Gamma_n)^{\Gamma_{n-1}}$, so in this case we have $(A,B,i,0)\in \mu^{-1}(0)^{\theta-ss}$.
In particular, we have a morphism $\iota:\tilde{X}^{reg}\rightarrow \mu^{-1}(\K \alpha)^{-1}/G=\tilde{X}$ given by
$x\mapsto G(A,B,i,j)$.

\begin{Lem}
The morphism $\iota$ is the usual inclusion  $\tilde{X}^{reg}\subset \tilde{X}$.
\end{Lem}
\begin{proof}
Indeed, $\iota|_{X^{reg}}$ is the inclusion,
this is how the identification of $X_0=V/\Gamma_n$ with $\mu^{-1}(0)\quo G$ is constructed. Now our
morphism $\tilde{X}^{ss}\rightarrow \tilde{X}$ gives rise to a $\K^\times$-equivariant morphism
$\tilde{X}_0\rightarrow \tilde{X}_0$ of schemes over $\K$ that is the identity in the zero fiber. So it is given
by $\exp(z^k \xi)$, where $\xi$ is a vector field of degree $-2k$ on $\tilde{X}_0$ (that is tangent to
the fibers of $\tilde{X}_0\twoheadrightarrow \K$) and does not vanish on $X_0$. However any vector field
on $X_0$ lifts to a vector field on $V$ and there are no nonzero vector fields of degree $\leqslant -2$ on $V$.
So $\xi=0$ and we are done. 
\end{proof}

By the construction of the morphism $\iota$, we get an identification of $\tilde{\Pro}^{\Gamma_{n-1}}$
with $\iota^*(\tilde{\mathcal{T}})$ on $\tilde{X}^{reg}$. But still the codimension of $\tilde{X}\setminus \tilde{X}^{reg}$
is $2$, and so this isomorphism extends to the whole variety $\tilde{X}$.

Now consider the case when $\theta\cdot\delta<0$. Then we have $\mu^{-1}(0)^{ss}\subset \Lambda_-$.
The $\Gamma_n$-module $\K\Gamma_n$ is self-dual and so we can also identify $\tilde{\Pro}_x^{\Gamma_{n-1}}$
with $(\K\Gamma_n)^{\Gamma_{n-1}*}$. Then we construct a morphism $\iota:\tilde{X}^{reg}\rightarrow \mu^{-1}(\K \alpha)^{-1}/G$
in the same way but now we are taking $1\in \tilde{\Pro}_x^{\Gamma_{n-1}}$ for $j$. For the same reasons
as before, $\iota$ is just the inclusion, and $\tilde{\Pro}^{\Gamma_{n-1}}=\iota^*(\tilde{\mathcal{T}})$.

\begin{Rem}
One can ask to describe $\Pro'^{\Gamma_{n-1}}$ for a Procesi bundle $\Pro'$ with different $\nu_{\Pro'}$.
Recall, see \cite{Maffei}, that, for $w\in W$, there is an isomorphism, say $\sigma_w:X^\theta\xrightarrow{\sim} X^{w\theta}$
of schemes over $X_0$. By the construction in \cite[6.4]{quant} (in particular, see the discussion after Lemma 6.4.3 in loc.cit.),
one has $\nu_{\sigma_{w*}\Pro}=w\nu_{\Pro}$. So $(\sigma_{w*}\Pro)^{\Gamma_{n-1}}=\sigma_{w*}\mathcal{T}$.

Now consider the bundle $\Pro':=\Pro^*$ that is also a Procesi bundle (we use a natural identification
of $S(V)\#\Gamma_n$ with its opposite, that is the identity on $V$ and
is the inversion on $\Gamma_n$). We claim that $\nu_{\Pro'}=w_0\sigma \nu_{\Pro}$,
where $w_0$ is the longest element in $W$, and $\sigma$ is a generator of $\ZZ/2\ZZ$. Indeed, we have an anti-automorphism
$\psi$ of $\Halg$ that is the identity on $V$, the inversion on $\Gamma_n$ and maps $h$ to $-h$,  $c(s)$ to $-c(s^{-1})$
(where $c(s)$ is a basis element in $\param$ corresponding to a symplectic reflection $s$) and so coincides with
$w_0\sigma$ on $\z$. The right module $\tilde{\Pro}'_h$ is obtained from the left module $(\tilde{\Pro}_h)^*$ by using the parity-antiautomorphism, see \cite[Section 2.2]{quant}, $\Dcal\xrightarrow{\sim}\Dcal^{opp}$. 
So the endomorphism algebra of $\tilde{\Pro}'_h$
is naturally identified with the opposite of $\End_{\Dcal^{opp}}(\tilde{\Pro}_h)$. Since $\Halg$ has no automorphisms
that are graded, preserve the parameter space and are the identity modulo the parameters, we see that
the identification $H_{\Pro'}\cong \Halg$ is obtained from
$H_{\Pro}\cong \Halg$ by applying the anti-automorphism of $\Halg$. This implies our claim,
because the antiautomorphism acts on $\param_{\Halg}$ as $w_0\sigma$.

Of course, if $\Pro^{\Gamma_{n-1}}=\mathcal{T}$, then $(\Pro^*)^{\Gamma_{n-1}}=\mathcal{T}^*$. Now we can characterize
$\Gamma_{n-1}$-invariants in all Procesi bundles on $X$.
\end{Rem}



\end{document}